\documentclass[11pt]{amsart}

\usepackage{amssymb,latexsym,amsmath,amsfonts}
\usepackage[mathscr]{eucal}
\usepackage{mathrsfs}
\usepackage{graphicx}
\usepackage{amscd}
\usepackage{amsthm}
\usepackage{amsxtra}
\usepackage[nointegrals]{wasysym}
\usepackage{bm}
\usepackage{calc}
\usepackage{float}
\usepackage[usenames,dvipsnames]{xcolor}
\usepackage{hyperref}
\usepackage{hyperref}
\hypersetup{colorlinks=true, citecolor=MidnightBlue, linkcolor=BrickRed}

\numberwithin{equation}{section}
\theoremstyle{definition}
\newtheorem*{definition}{Definition}
\theoremstyle{definition}

\newtheorem*{ntn}{Notation}

\theoremstyle{plain}
\newtheorem{theorem}{Theorem}[section]
\newtheorem{lemma}[theorem]{Lemma}

\newtheorem{Prop}[theorem]{Proposition}

\newcommand{\beas}{\begin{eqnarray*}}
\newcommand{\eeas}{\end{eqnarray*}}
\newcommand{\bes} {\begin{equation*}}
\newcommand{\ees} {\end{equation*}}
\newcommand{\be} {\begin{equation}}
\newcommand{\ee} {\end{equation}}
\newcommand{\bea} {\begin{eqnarray}}
\newcommand{\eea} {\end{eqnarray}}

\newcommand{\eps}{\varepsilon}

\newcommand{\de} {\delta}

\newcommand{\bdy}{\partial}
\newcommand{\Om}{\Omega}

\newcommand{\D}{\mathbb{D}}

\newcommand{\hol}{\mathcal{O}}

\newcommand{\ima}{\operatorname{Im}}

\newcommand{\vol}{\operatorname{vol}}

\newcommand\wt[1]{\widetilde{#1}}

\newcommand{\Bn} {\mathbb{B}^n}

\newcommand{\Cn}{\mathbb{C}^n}
\newcommand{\C} {\mathbb{C}} 

\newcommand{\rl}{\mathbb{R}}

\newcommand{\Rn} {\mathbb{R}^{n}}

\definecolor{ao(english)}{rgb}{0.0, 0.5, 0.0}
\begin{document}
\title[ Floating Bodies as Approximations of Bergman Sublevel Sets]{Convex Floating Bodies as Approximations of Bergman Sublevel Sets on Tube Domains}

\author{Purvi Gupta}
\address{Department of Mathematics, University of Western Ontario, London, Ontario N6A 5B7, Canada}
\email{pgupta45@uwo.ca}
%\begin{abstract}For a pseudoconvex tube domain, we prove estimates that relate the sublevel sets of its diagonal Bergman kernel to the floating bodies of its convex base. This allows us to associate a new affine invariant to any convex body.
%\end{abstract}
\maketitle
\section{Introduction}

The main objective of this paper is to establish a quantitative relationship between two collections of geometric objects associated with a given convex body in $\Rn$. One, its set of convex floating bodies --- an equiaffine-invariant construction studied by convex geometers, and the other, the collection of sublevel sets traced by the Bergman kernel of a tube domain over the given body. The latter is a natural object in complex analysis. Although, the bridge between convex and complex analysis on such domains has been exploited succesfully before --- Nazarov's paper \cite{Na12} is a noteworthy example --- the role of floating bodies in this interplay is yet to be explored.  Before we state our main result, we describe the central objects of this paper in some detail.
 
Let $D\subset\Rn$ be a bounded convex domain. For $\de>0$, its {\em convex floating body} $D_\de$ is the intersection of all the half-spaces whose defining hyperplanes cut off a set of volume $\de$ from $D$. Specifically, if $A$ denotes the set of all $(v,t)\in\Rn\times\Rn$ such that $\vol\{x\in D:x\cdot v\geq  t\}=\de$, then 
	\be\label{def_float}
		D_\de=\bigcap_{(v,t)\in A}\{x\in\Rn:x\cdot v< t\}.
	\ee
These are strictly convex and exhaust $D$ as $\de$ approaches zero. Inspired by a construction due to Dupin, these were first introduced by Sch{\"u}tt and Werner (in \cite{ScWe90}) as a tool for extending the notion of Blaschke's surface area measure to nonsmooth convex boundaries. Since its introduction, the floating body has made appearances in the context of polyhedral approximations (see \cite{Sc91}), the homethety conjecture (see \cite{ScWe94} and \cite{St06}) and, more recently, the hyperplane conjecture (in \cite{Fr12}). 

Now, let $\Om:=\{x+iy\in\Cn:y\in D\}$. Then, $\Om$ is a pseudoconvex tube domain in $\Cn$. The Bergman kernel of $\Om$, $K_\Om:\Om\times\Om\rightarrow\C$, is the reproducing kernel of the Bergman space $A(\Om)$ --- i.e., the space of Lebesgue square-integrable holomorpic functions on $\Om$, with the $L^2$-norm. It is known that $A(\Om)$ is nonempty, consists of Fourier-Laplace transforms of certain functions on $\Rn$, and 
	\be\label{eq_berg}
		K_\Om(z,w)=\frac{1}{(2\pi)^n}\int_{\Rn} \frac{e^{i(z-\overline w)\cdot t}}{\int_De^{-2x\cdot t}d\mu(x)}d\mu(t),
	\ee
where $\mu$ denotes the Lebesgue measure on $\Rn$ (see \cite{Sa88}, and references therein). Estimates for the Bergman kernel and associated quantities are of great interest to complex analysts and are the subject of many works. We will focus on the sets 
	\be\label{def_berg}
		D^M:=\{x\in D: K_D(x):=K_\Om(ix,ix)< M\}.
	\ee
These are strongly convex (this follows from the strict plurisubharmonicity of $\log K_\Om(z,z)$) and exhaust $D$ as $M\rightarrow\infty$ (as discussed in Section \ref{sec_proofs}). Although \eqref{eq_berg} gives a formula for $K_D(x)$, it can be hard to compute, even for some very simple examples (such as planar triangles). On the other hand, the convex floating bodies are simpler to construct and visualize. This is part of our motivation for establishing the following relation:

\begin{theorem}\label{thm_main}Let $D\subseteq\Rn$ be a bounded convex domain. Let $D_\de$ and $D^M$ be the $\de$-convex floating body and the Bergman $M$-sublevel set of $D$, respectively (see \eqref{def_float} and \eqref{def_berg}). Then, there exist dimensional constants $\ell_n>0$ and $u_n>0$ such that
	\bes 
    		D^{\ell_n\de^{-2}}\subseteq D_\de\subseteq D^{u_n\de^{-2}}
    \ees
for small enough $\de$.
\end{theorem}
Another reason to compare these two collections is their suitability for the following scheme. Suppose $G$ is a group of volume-preserving transformations that acts on $\Rn$ (or $\Cn$), and $D\subset\Rn$ (or $\Cn$) is a bounded domain. If $\{D(\eps)\}_{\eps>0}$ is a $G$-invariant collection of exhausting subsets of $D$, and 
	\bes
		\vol(D\setminus D(\eps))\sim f(\eps)\qquad\ \text{as}\ \eps\rightarrow 0,
	\ees
 for some continuous $f$ with $f(0)=0$, then the weak-$*$ limit (if it exists) of $f(\eps)^{-1}$ times the Lebesgue measure on $D\setminus D(\eps)$ yields a $G$-invariant measure on $\bdy D$. If $D$ is strongly convex and $D(\eps)$ is chosen as the convex floating body $D_\eps$, then this measure is the normalized affine surface area measure on $\bdy D$ (this is implicit in Sch{\"u}tt and Werner's paper \cite{ScWe90}). For other convex domains, the floating bodies can lead to `lower-dimensional' affine measures (for instance, this measure is supported on the vertices in the case of polygons --- see \cite{Sc91}). If the above scheme is carried out for a strongly pseudoconvex domain $\Om\Subset\Cn$, using the Bergman sublevel sets, then one obtains the normalized Fefferman hypersurface measure on $\bdy\Om$ (see \cite[Prop. 1.5]{Gu15}). If $D$ is strongly convex, the tube domain $\Om:=\Rn+iD$ is strongly pseudoconvex, and the Fefferman measure on $\bdy\Om$ reduces to the affine measure along $\bdy D$. It follows that if $D(\eps)$ is set as the Bergman sublevel set $D^{1/\eps}$, then, again we obtain the normalized affine surface area measure on $\bdy D$. Theorem \ref{thm_main} implies that, analogous to strongly convex domains, the two competing classes $\{D_\eps\}$ and $\{D^{1/\eps}\}$ will yield comparable equiaffine-invariant measures on $\bdy D$ for a general convex domain $D\subset\Rn$. This is surprising since in the absence of strong convexity, we do not have any Sch{\"u}tt-Werner or H{\"o}rmander-type estimates relating these sets to the curvature of $\bdy D$ (the estimates referrred to are used in the proof of Proposition \ref{prop_strcvx}). 

The rest of the article is organized as follows. We provide a proof of Theorem \ref{thm_main} in the next section. The constants $\ell_n$ and $u_n$ are computed therein. In Section \ref{sec_exam}, we set up a new affine-invariant constant associated to a convex body, and compute it for some examples. At the end, we indicate some possible avenues of future exploration.   

\noindent {\bf Acknowledgement.}	The author would like to thank David Barrett who encouraged her to explore this problem in the context of tube domains, and supported this work with lots of feedback. 

\section{Proof of Theorem \ref{thm_main}}\label{sec_proofs}

\begin{ntn}
We first clarify some notation that will appear throughout the rest of this article. We use $\mathbb{B}^n$ and $\omega_n$ to denote the unit Euclidean ball and its volume, respectively, in $\Rn$. The unit disc in $\C$ is written as $\D$. The space of holomorphic maps from $D_1$ to $D_2$ is denoted by $\hol(D_1;D_2)$. For complex-valued $n$-tuples $a=(a_1,...,a_n)$ and $b=(b_1,...,b_n)$, $a\cdot b=a_1b_1+\cdots+a_nb_n$.   
\end{ntn}

We now briefly argue the fact that the Bergman sublevel sets $\{D^M\}_{M>0}$ exhaust $D$. Although, this is not necessary for our main proof, it is an essential feature of the comparison we are making between $\{D_\eps\}_{\eps>0}$ and $\{D^M\}_{M>0}$. 
\begin{lemma}
Let $D\subset\Rn$ be a bounded convex domain. Then, for any $x_0\in\bdy\Om$, $K_D(x)\rightarrow\infty$ as $x\rightarrow x_0$.
\end{lemma}
\begin{proof} Let $R:=\{(z_1,...,z_n):(\log|z_1|,...,\log|z_n|)\in D\}$. As $D$ is a bounded convex domain, $R$ is a bounded pseudoconvex Reinhardt domain in $\Cn$ that satisfies the Fu condition --- i.e., it does not intersect any complex hyperplane of the form $\{(z_1,...,z_n)\in\Cn:z_j=0\}$. Thus, $R$ is hyperconvex, and $K_R(z,z)\rightarrow\infty$ as $z\mapsto z_0$, for any $z_0\in\bdy R$ (these are results from \cite{Zw} and \cite{Oh84}, respectively). Now, by Theorem $2$ and estimate $(7)$ in Fu's paper \cite{Fu01},
	\beas
		K_R\big((e^{x_1},...,e^{x_n}),(e^{x_1},...,e^{x_n})\big)={e^{-2(x_1+\cdots+x_n)}}\sum_{k\in\mathbb{Z}^n}K_{\Rn+iD}(ix,ix+2k\pi)\\
		\leq CK_{\Rn+iD}(ix,ix)\sum_{\substack{k\in\mathbb{Z}^n,\\ k\neq(0,\cdots,0)}}\frac{1}{|k|^2},
	\eeas
for $x=(x_1,...,x_n)\in D$, and some constant $C$ independent of $x$. Thus, $K_D(x)\geq \wt C K_R\big((e^{x_1},...,e^{x_n}),(e^{x_1},...,e^{x_n})\big)$, where $\wt C$ is independent of $x$. Combining this with the hyperconvexity of $R$, we get the desired result. 
\end{proof}

We now proceed to the proof of our main theorem. We rely on Nazarov's approach from \cite{Na12}, the main source of challenge being the lack of any symmetry assumptions on $D$.

\noindent{\em Proof of Theorem \ref{thm_main}}. Let $D$ be a bounded convex domain in $\Rn$. We first establish the existence of $u_n.$ For this, we repeat an estimate due to Nazarov (see \cite[Section~3]{Na12}). Let $E\subset\Rn$ be an origin-symmetric convex body. One uses formula \eqref{eq_berg} to write
	\be\label{eq_bergsym}
		K_E(0)=\frac{1}{(2\pi)^n}\int_{\Rn}\frac{1}{J_E(t)}d\mu(t)
	\ee
where 
	\bes
		J_E(t)=\int_Ee^{-2x\cdot t}d\mu(x).
	\ees
Fix a $y\in E$. Then, $E^y:=\frac{1}{2}(y+E)\subseteq E$. So, we obtain
	\bea		
		J_E(t)\geq \int_{E^y}e^{-2x\cdot t}d\mu(x)&=&2^{-n}\int_Ee^{-2(\frac{v+y}{2})\cdot t}d\mu(v)\notag \\
			&=&2^{-n}e^{-y\cdot t}\int_Ee^{-v\cdot t}d\mu(v)\notag \\
			&\geq & 2^{-n}e^{-y\cdot t}\vol(E),\label{eq_naz_ineq}
	\eea	
where we use the convexity of $v\mapsto e^{-v\cdot t}$ on $E$ for every $t$, and the observation that any convex function $f$ on $E$ satisfies 
\bes 
	\int_Ef(x)d\mu(x)\geq f(0)\vol(E)
\ees
by the symmetry of $E$. Next, recall that the polar body of $E$ is given by $E^\circ=\{y\in\Rn:x\cdot y\leq 1\ \text{for all}\ x\in E\}$, and 
	\beas	
		||x||_{E^\circ}&:=&\min\{\alpha>0:x\in\alpha E^\circ\}\\	
							&=&\max\{x\cdot y:y\in E\}.
	\eeas
So, maximizing \eqref{eq_naz_ineq} over all $y\in E$, we obtain that 
	\bes
		J_E(t)\geq 2^{-n}e^{||-t||_{E^\circ}}\vol(E)=2^{-n}e^{||t||_{E^\circ}}\vol(E),
	\ees
for all $t\in\Rn$. Substituting this back in \eqref{eq_bergsym}, we see that 
	\bea
		K_E(0)&\leq& \frac{1}{\pi^n\vol(E)}\int_{\Rn}e^{-||t||_{E^\circ}}d\mu(t)\notag \\
				&=& \frac{1}{\pi^n\vol(E)}\int_{\Rn}\int_{s\geq {||t||_{E^\circ}}}e^{-s}ds\: d\mu(t)\notag\\	
				&=&\frac{1}{\pi^n\vol(E)}\int_0^\infty e^{-s}\int_{\{t\in\Rn:||t||_{E^\circ}\leq s\}}d\mu(t)\:ds\notag\\
				&=&\frac{\vol(E^\circ)}{\pi^n\vol(E)}	\int_0^\infty s^ne^{-s}ds=\frac{n!\vol(E^\circ)}{\pi^n\vol(E)}.
				\label{eq_nazupp}
	\eea

Now, we return to $D$. Fix a positive $\de<<\vol(D)$. For each $v\in S^{n-1}$, let $r_v$ denote the unique real number such that 
	\bes 
    		\vol(\{x\in D:x\cdot v > r_v\})=\de.
    \ees
Set  $H_v:=\{x\in\Rn:x\cdot v=r_v\}$ and $D|_v:=\{x\in D:x\cdot v >r_v\}$. $D|_v$ is a continuous family of convex domains in $D$, each of volume $\de$. We let $E_v$ denote the circumscribed L{\"o}wner-John ellipsoid of $D|_v$ --- i.e., the unique ellipsoid of minimal volume that contains $D|_v$ (see \cite[Lecture 3]{Ba97}, for more on L{\"o}wner-John ellipsoids). Then, due to a result by F. John (\cite{Jo2014}), if 
	\bes
		E_v=c_v+A_v(\Bn),
      \ees
for some $A_v\in \operatorname{GL}(n;\rl)$, then on shrinking,
	\bes
		E_v^n:=c_v+\frac{1}{n}A_v(\mathbb{B}_n)\subseteq D|_v.
	\ees
In particular, for every $v\in S^{n-1}$, 
	\be\label{eq_vol}
    	\vol(E^n_v)=\frac{1}{n^n}\vol(E_v)\geq\frac{1}{n^n}\vol(D|_v)=\frac{\de}{n^n}.
    \ee

We now estimate the Bergman kernel of $D$ at each $c_v$. We first observe that since the Bergman kernel is invariant under translations, $K_{\frac{1}{n}A_v(\Bn)}(0)=K_{E^n_v}(c_v)$ for each $v\in S^{n-1}$. But, since $\frac{1}{n}A_v(\Bn)$ is an origin-symmetric convex domain in $\Rn$, we get by \eqref{eq_nazupp} that
		\be\label{eq_naz} 	
    	K_{E^n_v}(c_v)=K_{\frac{1}{n}A_v(\Bn)}(0)\leq \frac{n!\vol\Big(\left(\frac{1}{n}A_v(\Bn)\right)^\circ\Big)}{\pi^n\vol\Big(\frac{1}{n}A_v(\Bn)\Big)}.
    \ee 
This can be combined with the Blaschke-Santal{\'o} inequality for origin-symmetric convex bodies:
	\bes 
    	\vol(D^\circ)\vol(D)\leq (\omega_n)^2,
     \ees 
and \eqref{eq_vol}, to obtain that 
	\bes 
		K_{E^n_v}(c_v)\leq\frac{n!(\omega_n)^2}{\pi^n\vol\left(\tfrac{1}{n}A_v(\Bn)\right)^{2}}= \frac{n!(\omega_n)^2}{\pi^n\vol(E^{n}_v)^{2}}
			\leq \frac{n!n^{2n}(\omega_n)^2}{\pi^n\de^2}.
		\ees
Since $c_v\in E^n_v\subseteq D|_v\subset D$, by the monotonicity of the Bergman kernel, 
      	\be\label{eq_mon}
        	 K_D(c_v)\leq K_{E^n_v}(c_v)
				\leq u_n\de^{-2} ,\qquad \text{for every}\ v\in S^{n-1},
          \ee
where $u_n:=\dfrac{n!n^{2n}(\omega_n)^2}{\pi^n}$. 

Now, we claim that the image of the map $\gamma:S^{n-1}\rightarrow D\setminus D_\de$ given by $v\mapsto c_v$ `surrounds' $D_\de$ --- i.e., $D_\de$ is contained in an open set $U$ such that $\bdy U\subseteq\gamma(S^{n-1})$. 
Our argument is as follows. Let $b_v$ denote the barycenter of $H_v\cap D$. Then, by Lemma $2$ in \cite{ScWe94}, every $x\in\bdy D_\de$ coinicides with a $b_v$ for some $v\in S^{n-1}$. Thus, the image of the map $\beta:S^{n-1}\mapsto D\setminus D_\de$ given by $v\mapsto b_v$ surrounds $D_\de$ (in the sense described above --- in fact, $U=D_\de$ in this case). Now, $T:S^{n-1}\times[0,1]\mapsto D\setminus D_\de$ given by $(v,t)\mapsto (1-t)b_v+tc_v$ is a homotopy between $\beta(S^{n-1})$ and $\gamma(S^{n-1})$ whose image is entirely contained in the complement of $D_\de$. Thus, $\gamma(S^{n-1})$ must surround $D_\de$ as well, and there is an open set $U\subset D$, such that $U\supseteq D_\de$ and $\bdy U\subseteq \gamma(S^{n-1})$. Thus, by the maximum principle ($x\mapsto \log K_D(x)$ is strongly convex on $D$),
	\bes 
		\sup_{x\in D_\de}K_D(x)\leq\sup_{y\in U} K_D(y)\leq\sup_{y\in \bdy U} K_D(y)\leq \sup_{v\in S^{n-1}}K_D(c_v)\leq \frac{u_n}{\de^2}.
	\ees
This shows that $D_\de\subset D^{u_n\de^{-2}}$.

We now turn to the existence of $\ell_n$. Once again, we fix $\de$ so small that $D_\de$ is nonempty, and $H_v$ is as before. It suffices to show that for any $x\in H_v\cap D$, $K_D(x)\geq \ell_n\de^{-2}$ for some $\ell_n>0$ independent of $v$, $\de$ and $D$. This is because for any $x\in\bdy D_\de$, there is a supporting hyperplane of $D_\de$ that cuts off a set of volume $\de$ from $D$ --- i.e., there is a $v\in S^{n-1}$ such that $x\in H_v\cap D$ (see Lemma $2$ in \cite{ScWe94}). The required estimate will be obtained from the following lower bound for convex domains due to B{\l}ocki in \cite{Bl14}:
	\be\label{eq_kob}
		K_\Om(w,w)\geq \frac{1}{\vol_{\Cn}(I_\Om(w))}, \qquad w\in\Om,
	\ee
where $I_\Om(w)\subset\Cn$ is the Kobayashi indicatrix of $\Om$ given by 
	\bes
		I_\Om(w)=\{\phi'(0):\phi\in\hol(\D;\Om),\phi(0)=w\}.
	\ees
For us, $\Om:=\Rn+iD$, and $w=ix$ for some $x\in H_v\cap D$. We are seeking an upper bound on $\vol_{\Cn}(I_\Om(w))$.  

Without loss of generality, we assume that $x$ is the origin in $\Rn$ and $v=(0,...,0,1)$. In particular, $H_v$ is the hyperplane $\{(x_1,...,x_n)\in\Rn:x_n=0\}$ and $D\cap\{x_n>0\}=D|_v$. We will follow Nazarov's technique from \cite{Na12} (as used by B{\l}ocki in \cite{Bl14}). We recall that $D^\circ=\{u\in\Rn:y\cdot u\leq 1\ \text{for all}\ y\in D\}$, which is the same as $\{u\in\Rn:y\cdot u< 1\ \text{for all}\ y\in D\}$ since $D$ is open. Now, consider the half-plane $S:=\{z\in \C:\ima z< 1\}$, and let $\Phi:S\mapsto \D$ denote the conformal map $z\mapsto-iz/{(z-2i)}$. Then, $\Phi(0)=0$ and $\Phi '(0)=1/2$. For a fixed $u\in D^\circ$ and any $\phi\in\hol(\D;\Om)$ such that $\phi(0)=w$, the map $F:z\mapsto \Phi(\phi(z)\cdot u)$ is a holomorphic self-map of $\D$ that fixes the origin (since we are assuming that $w$ is the origin in $\Cn$). Thus, by the Schwarz lemma, $|F'(0)|\leq 1$, or $|\phi '(0)\cdot u|\leq 2$. So, $\frac{1}{2}I_\Om(w)\subseteq D_\C$, where
	\bes
		D_\C:=\{z\in\Cn:|z\cdot u|\leq 1\ \text{for all}\ u\in D^\circ\}. 
	\ees 
Note that $D_\C\subseteq \Big(\hat D\cup(-\hat D)\Big)+i\Big(\hat D\cup(-\hat D)\Big)$, where
	\bes
		\hat D=\{(x_1,...,x_n)\in\Rn:|x\cdot u|\leq 1\ \text{for all}\ u\in D^\circ, x_n\geq 0\}.
	\ees
 But, 
		\bes 
		\hat D\subseteq \overline{D}\cap\{x\in\rl^n:x_n\geq 0\}\subseteq \overline{D|_v},
		\ees
and $\vol(D|_v)=\de$. Thus, recalling \eqref{eq_kob}, 
	\bes 
		K_D(x)=K_\Om(w,w)\geq \frac{1}{\vol_{\Cn}(I_\Om(w))}\geq \left(2\right)^{-2n}(2\de)^{-2}.
	\ees
Therefore, $D_\de\supseteq D^{\ell_n\de^{-2}}$, where $\ell_n:=\dfrac{1}{4^{n+1}}$. This completes the proof of Theorem \ref{thm_main}.

\qed

\section{A new affine invariant and some examples}\label{sec_exam}

It is unlikely that the values of $\ell_n$ and $u_n$ computed above are optimal. For one, John's theorem on L{\"o}wner-John ellipsoids can be replaced by results that utilize other centrally-symmetric bodies, perhaps yielding better bounds. However, we believe optimal bounds can be obtained if we restrict ourselves to certain classes of convex bodies. Before we support this claim with some computations, we associate a new quantity $\theta_D$ to a convex body. 

\begin{definition}
Suppose $D\subset\Rn$ is a convex body. Let
		\beas
			\ell_D&:=&\liminf_{\de\rightarrow 0}\Big(\sup\{\ell>0: D^{\ell/\de^{2}}\subseteq D_\de\}\Big);\\
			u_D&:=&\limsup_{\de\rightarrow 0}\Big(\inf\{u>0: D_\de\subseteq D^{u/\de^{2}}\}\Big);\\
			\theta_D&:=&\frac{\ell_D}{u_D}.
		\eeas
\end{definition}

We establish some properties of $\theta_D$. 
\begin{Prop}\label{prop_theta} For a convex body $D\subset\Rn$,
\begin{enumerate}
\item $\dfrac{\pi^{n}}{n!n^{2n}4^{n+1}(\omega_n)^2}\leq \theta_D\leq 1$	.
\item  $\theta_D$ is affine invariant, i.e., $\theta_D=\theta_{A(D)}$ for any affine map $A$ on $\Rn$.
\end{enumerate} 
\end{Prop}
\begin{proof} $(1)$ The upper bound on $\theta_D$ follows from the fact that $\ell_D\leq u_D$, by definition. The lower bound is a consequence of Theorem \ref{thm_main}, where we have essentially shown that $\ell_D\geq{1}/{4^{n+1}}$ and $u_D\leq n!n^{2n}(\omega_n)^2/\pi^n$.   

$(2)$ The affine invariance of $\theta_D$ follows from that of $\ell_D$ and $u_D$, which, in turn, is a consequence of the transformation properties of $D_\de$ and $D^M$ under affine maps. More concretely, if $A:\Rn\rightarrow\Rn$ is an affine map and $H$ is a hyperplane that cuts off a set of volume $\de$ from $D$, then the hyperplane $A(H)$ cuts off a set of volume $|\det(A)|\de$ from $A(D)$. Therefore, 
	\be\label{eq_floattran} 
		A(D_\de)=A(D)_{|\det A|\de}, \qquad\text{for all}\ \de>0.
	\ee
Now, let $\Om:=\Rn+iD$ and $A_\C$ be the map $z\mapsto Az$. Then, $A_\C:\Cn\rightarrow\Cn$ is a biholomorphic map with $\operatorname{Jac}_\C A_\C=\det A$, where $\operatorname{Jac}_\C$ denotes the complex Jacobian. We use the well-known fact that the Bergman kernel of $\Om$ satisfies
	\bes
		K_{A_\C(\Om)}(A_\C(z),A_\C(z))|\operatorname{Jac}_\C(A_\C)|^2=K_\Om(z,z),\qquad \text{for all}\ z\in \Om.
	\ees
Hence, 
	\be\label{eq_bergtran}
		A(D^M)= A(D)^{M/|\det A|^2}.
	\ee
Combining \eqref{eq_floattran} and \eqref{eq_bergtran}, we see that if $D^{\ell/\de^2}\subseteq D_\de\subseteq D^{u/\de^2}$, then $A(D)^{\ell/(|\det A|\de)^2}\subseteq D_{|\det A|\de}\subseteq D^{u/(|\det A|\de)^2}$. Hence, the affine invariance of $\ell_D$, $u_D$ and $\theta_D$. 
\end{proof}

We now compute some examples to indicate the extent to which $\theta_D$ distinguishes convex domains. 

\begin{Prop}\label{prop_strcvx}
If $D$ is strongly convex --- i.e., the second fundamental form on $\bdy D$ is positive definite everywhere on $\bdy D$ --- then, $\theta_D=1$. 
\end{Prop}
\begin{proof} We begin with some notation (see Figure \ref{fig_scheme}). For $x\in\bdy D$, let $N(x)$ be the unique outer unit normal to $\bdy D$ at $x$, and $H(x)=\{y\in\Rn:y\cdot N(x)=x\cdot N(x)\}$. For $\de>0$, let $\Delta(x,\de)$ denote the width of the slice of volume $\de$ cut off by a hyperplane $H(x,\de)$ perpendicular to $N(x)$ --- i.e.,
\bes
	\vol\{y\in D:y\cdot N(x)>x\cdot N(x)-\Delta(x,\de)\}=\de. 
\ees 
and
\beas
	H(x,\de)&=&\{y\in\Rn:y\cdot N(x)=x\cdot N(x)-\Delta(x,\de)\}\\&=&H(x)-\Delta(x,\de)N(x).
\eeas
Let $x^\de$ denote the barycenter of $H(x,\de)\cap D$. 

\begin{figure}[H]
\centering
\resizebox{3in}{!}{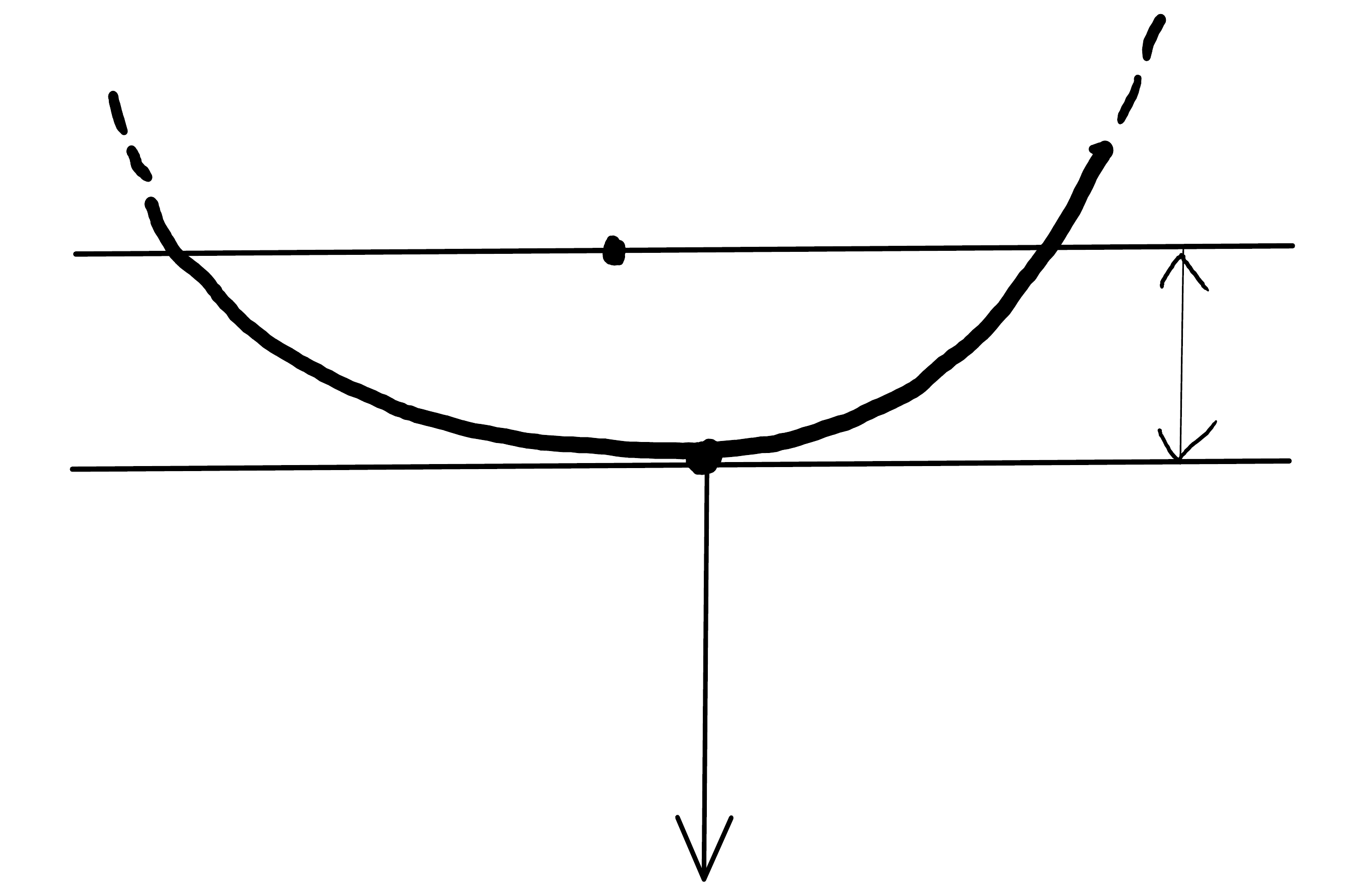}
\caption{}\label{fig_scheme}
\end{figure}

Now --- as $\theta_D$ is affine invariant --- for a fixed $x_0\in\bdy D$, we can choose affine co-ordinates, so that $x_0$ is the origin, the outer unit normal $N(x_0)=(0,...,0,-1)$ and $H(x_0,\de)=\{(x',y):y=\Delta(x_0,\de)\}$, where $x'=(x_1,...,x_{n-1})$. There is a neighborhood $U_0$ of $x_0$ such that $U_0\cap D=\{y>\phi(x')\}$, where $\phi:\rl^{n-1}\rightarrow\rl$ is a convex function of the form
	\bes
		\phi(x)=\alpha(x_1^2+\cdots+x_{n-1}^2)+ \operatorname{h.o.t.},	
	\ees
for some $\alpha>0$. Thus, each $H(x_0,\de)\cap D$ satisfies the equation
	\bes 
		\Delta(x_0,\de)=\alpha(x_1^2+\cdots+x_{n-1}^2)+ \operatorname{h.o.t.}
	\ees
in the hyperplane $y=\Delta(x_0,\de)$. So, we may estimate the the barycenter of $H(x_0,\de)\cap D$ as 
\bes 
	x_0^\de=(o(\sqrt{\Delta(x_0,\de)}),...,o(\sqrt{\Delta(x_0,\de)}),\Delta(x_0,\de))\qquad \text{as}\ \de\rightarrow 0.
\ees 
Thus, minimizing $\operatorname{dist}(x_0^\de,z	)$ over all $z\in\bdy\Om$, we obtain that
\be\label{eq_bary}
\lim_{\de\rightarrow 0}\frac{\Delta(x_0,\de)}{\operatorname{dist}(x_0^\de,\bdy D)}=1.
\ee
Moreover, using Dupin indicatrices (see \cite[Lemma 10]{ScWe90}), it is known that
	\be\label{eq_scwe}
		\lim_{\de\rightarrow 0}\frac{\Delta(x_0,\de)^{n+1}}{\de^2}=\frac{1}{2^{n+1}}\left(\frac{n+1}{\omega_{n-1}}\right)^2\kappa(x_0),
	\ee
where $\kappa$ is the Gaussian curvature function of $\bdy D$. Lastly, since $\Om$ is strongly convex, $\Om=\Rn+iD$ is strongly pseudoconvex. Thus, by H{\"o}rmander's estimate (in \cite{Ho65}), we have that 	
	\be\label{eq_horm}
		\lim_{x\rightarrow x_0\in\bdy D}\operatorname{dist}(x,\bdy D)^{n+1}K_D(x)=\frac{n!}{(4\pi)^n}\kappa(x_0). 
	\ee	
Since, $\lim_{\de\rightarrow 0}x_0^\de=x_0$, we can combine \eqref{eq_bary}, \eqref{eq_scwe} and \eqref{eq_horm} to obtain that 
	\bes 
		\lim_{\de\rightarrow 0}\de^2 K_D(x_0^\de)=\frac{n!2^{n+1}}{(4\pi)^n}\left(\frac{\omega_{n-1}}{n+1}\right)^2=:a_n.
	\ees
Hence, $(x_0,\de)\mapsto \de^2K_D(x_0^\de)$ extends to a (uniformly) continuous function on $\bdy D\times [\hat\de,0]$. So, given $\eps>0$, there is a $\de_\eps>0$ such that for $\de<\de_\eps$, 
	\bes
		\frac{a_n-\eps}{\de^2}	<K_D(x^\de)<\frac{a_n+\eps}{\de^2},	\qquad\ \text{for all}\ x\in\bdy D.
	\ees	
According to Lemma $2$ in \cite{ScWe94}, each $y\in\bdy D_\de$ is the barycenter $x^\de$ of some $H(x,\de)\cap D$. Therefore, for $\de<\de_\eps$,
	\bes 
		D^{(a_n-\eps)\de^{-2}}\subset D_\de\subset D^{(a_n+\eps)\de^{-2}}.
	\ees 
Thus,  
	\bes
			\theta_D
				=\frac{\liminf_{\de\rightarrow 0}\sup\{\ell>0:  D^{\ell/\de^{2}}\subseteq D_\de\}}
						{\limsup_{\de\rightarrow 0}\inf\{u>0: D_\de\subseteq D^{u/\de^{2}}\}}
				\geq \frac{a_n-\eps}{a_n+\eps}
						.
		\ees
Since $\eps>0$ is arbitrary, and $\theta_D\leq 1$, our claim follows. 
\end{proof}

We contrast the above example with the next one, where the Gaussian curvature of the boundary vanishes on a large part of it. 

\begin{Prop}\label{prop_poly}
Let $D\Subset\rl^2$ be a triangle or a parallelogram. Then, $\theta_D={4}/{\pi^2}$.
\end{Prop}

\begin{proof} As all planar triangles and parallelograms are affine images of the triangle $T=\{(x,y)\in\rl^2:x>0,y>0,x+y<1\}$
 and the square $S:=(0,1)\times (0,1)$, respectively, it suffices to show that $\ell_S=\ell_T$, $u_S=u_T$ and $\theta_S=4/\pi^2$. We take this approach as it is hard to directly compute $\theta_T$.

We start with a description of the floating body of $S$. For small enough $\de>0$, the boundary of $S_\de$ is a piecewise smooth curve, each smooth piece of which is a part of a hyperbola (see Figure \ref{fig_sqfloat}). Specifically, 
	\bes
		S_\de=\left\{(x,y)\in\rl^2:\min\Big(xy,(1-x)y,x(1-y),(1-x)(1-y)\Big)>\de/2\right\}.
	\ees 
\begin{figure}[H]
\centering
\resizebox{2.2in}{!}{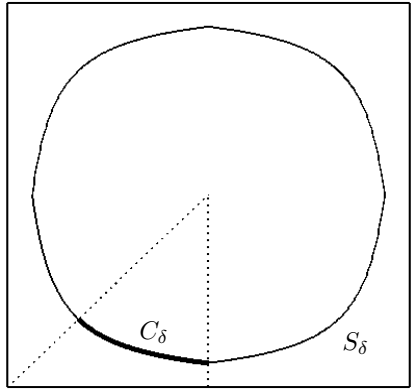}
\caption{A convex floating body for $S$.}
\label{fig_sqfloat}
\end{figure}
Due to the eight-fold symmetry of $S$, we will focus on the one-eighth part of the boundary given by $C_\de:=\bdy S_\de\cap \{(x,y):0\leq y\leq x\leq 1/2\}$ (thickened in Figure \ref{fig_sqfloat}). For $\de<<1/2$, $C_\de$ can be parametrized as
	\bes 
		t\mapsto c(t):=\left(t,\frac{\de}{2t}\right),\qquad\ \sqrt{\frac{\de}{2}}\leq t\leq \frac{1}{2}.
	\ees
To estimate $K_S$ on $C_\de$, we observe that 
	\bes 
		C_\de\subset T\subset S\subset \wt T\qquad \text{(see Figure \ref{fig_sqtri})},
	\ees
where $\wt T$ is the image of $T$ under the map $(x,y)\mapsto (2x,2y)$. 
\begin{figure}[H]
\centering
\resizebox{2.2in}{!}{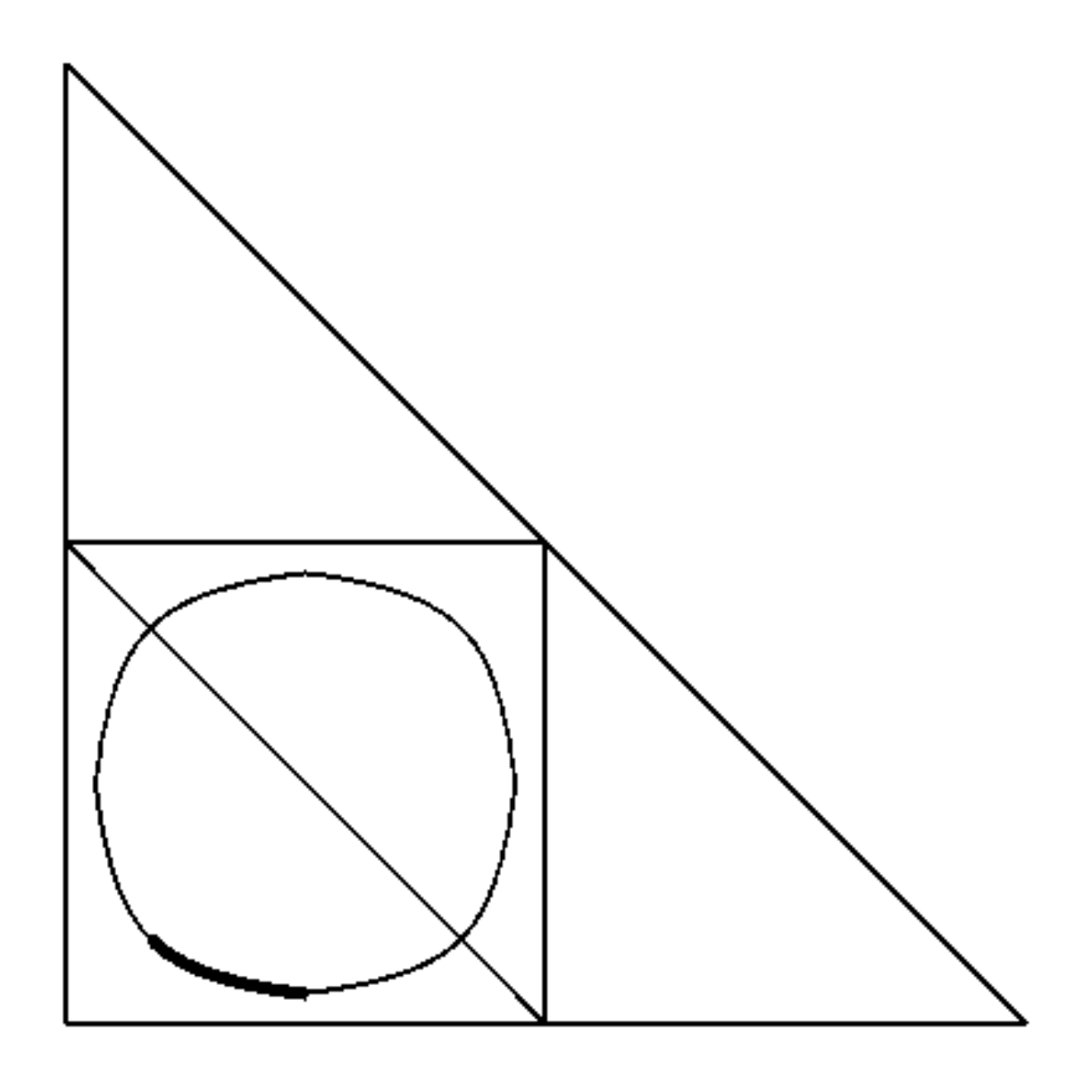}
\caption{}
\label{fig_sqtri}
\end{figure}

The descriptions of the floating bodies of $T$ and $\wt T$ are also needed:
\beas
		T_\de&=&\left\{(x,y)\in\rl^2:\min\Big(xy,(1-x-y)y,(1-x-y)x\Big)>\de/2\right\};\\
		\wt T_\de&=&\left\{(x,y)\in\rl^2:\min\Big(xy,(2-x-y)y,(2-x-y)x\Big)>\de/2\right\}.
	\eeas
These explicit descriptions allow us to conclude that, for $\de<<1/2$, 
	\be\label{eq_sqtri}
	 C_\de\subset T_{\de-2\de^2}\quad \text{and}\quad  C_\de\subset \bdy\wt T_\de\subset \wt T\setminus \wt T_{\de+2\de^2}.
	\ee

Now fix an arbitrary $\eps>0$. Then, for small enough $\de$,
\begin{enumerate}
\item $(1-\eps)\de<\de-2\de^2$ and $\de+2\de^2<(1+\eps)\de$; and 
\item $T_\de\subseteq T^{(1+\eps)u_T\de^{-2}}$ and $\wt T^{(1-\eps)\ell_T\de^{-2}}\subseteq \wt T_\de$
\end{enumerate} 
The latter follows from the definitions of $\ell_D$ and $u_D$, and the fact that $\ell_{\wt T}=\ell_T$ due to affine invariance (established in the proof of Proposition \ref{prop_theta}). We combine \eqref{eq_sqtri}, $(1)$, the montonicity of $T_\de$ and $\wt T_\de$, and $(2)$ to conclude that: 
	\bes
		C_\de\subset T_{\de-2\de^2}\subset T_{(1-\eps)\de}\subset T^{(1+\eps)u_T(1-\eps)^{-2}\de^{-2}}
	\ees
and
		\bes
		C_\de\subset \wt T\setminus \wt T_{\de+2\de^2}\subset \wt T\setminus \wt T_{(1+\eps)\de}\subset \wt T\setminus \wt T^{(1-\eps)\ell_T(1+\eps)^{-2}\de^{-2}}.
	\ees
Thus, for all $c\in C_\de$,
	\bes 
		K_T(c)< \frac{(1+\eps)u_T}{(1-\eps)^2\de^2}\quad \text{and }\quad K_{\wt T}(c)>\frac{(1-\eps)\ell_T}{(1+\eps)^2\de^2}.
	\ees
So, by the montonocity of the Bergman kernel, 
	\bes 
		\frac{(1-\eps)\ell_T}{(1+\eps)^2\de^2}<K_{\wt T}(c)<K_S(c)<K_T(c)< \frac{(1+\eps)u_T}{(1-\eps)^2\de^2}.
	\ees
As $\eps>0$ was arbitrarily chosen, and the estimates on $C_\de$ transfer to $\bdy S_\de$ due to symmetry,
	\bes
		S^{\ell_T\de^{-2}}\subset S_\de\subset S^{u_T\de^{-2}}. 
	\ees
Thus, $u_S\leq u_T$, $\ell_S\geq \ell_T$ and, consequently, $\theta_S\geq \theta_T$. An analogous computation can be executed after switching the roles of $S$ and $T$ to obtain that $\theta_T\geq \theta_S$, thus yielding the desired equality. It now suffices to compute $\theta_S$. 

We use \eqref{eq_berg} to compute the Bergman kernel of $\rl^2+iS$ at any point $(x,y)\in S$:
	\bes	
		K_S\big((x,y)\big)=\frac{\pi^2}{16}\csc^2(\pi x)\csc^2(\pi y).
	\ees
Once again, we can exploit the symmetry of $S$ to obtain that 
	\beas 
		\ell_S&=&\lim_{\de\rightarrow 0}\inf_{c\in C_\de}K_S(c)\de^2
				=\lim_{\de\rightarrow 0}\inf_{t\in[\sqrt{\de/2},1/2]}\frac{\pi^2\de^2}{16}\csc^2(\pi t)\csc^2\left(\frac{\pi\de}{2t}\right)=\frac{1}{4\pi^2};\\
		u_S&:=&\lim_{\de\rightarrow 0}\sup_{c\in C_\de}K_S(c)\de^2
				=\lim_{\de\rightarrow 0}\sup_{t\in[\sqrt{\de/2},1/2]}\frac{\pi^2\de^2}{16}\csc^2(\pi t)\csc^2\left(\frac{\pi\de}{2t}\right)=\frac{1}{16}.
	\eeas
Therefore, $\theta_T=\theta_S=4/\pi^2$. 
\end{proof}
We strongly suspect that $\theta_D=1$ completely characterizes strongly convex bodies, and that Proposition \ref{prop_poly} can be extended to all planar convex polygons. In fact, we believe that, for $n=2$, these represent the two extremes of the range of values for $\theta_D$ (this would improve the first part of Proposition \ref{prop_theta}). Furthermore, it is likely that using the almost polygonal bodies constructed in \cite{ScWe92} one can construct planar convex bodies with any prescribed value of $\theta_D$ in the interval $(4/\pi^2,1)$.   

\bibliography{tubedomains}

\begin{thebibliography}{10}

\bibitem{Ba97}
Keith Ball.
\newblock An elementary introduction to modern convex geometry.
\newblock {\em Flavors of Geometry}, 31:1--58, 1997.

\bibitem{Bl14}
Zbigniew B{\l}ocki.
\newblock A lower bound for the {B}ergman kernel and the {B}ourgain-{M}ilman
  inequality.
\newblock In {\em Geometric Aspects of Functional Analysis}, pages 53--63.
  Springer, 2014.

\bibitem{Fr12}
Daniel Fresen.
\newblock The floating body and the hyperplane conjecture.
\newblock {\em Arch. Math.}, 98(4):389--397, 2012.

\bibitem{Fu01}
Siqi Fu.
\newblock Transformation formulas for the {B}ergman kernels and projections of
  {R}einhardt domains.
\newblock {\em Proc. Amer. Math. Soc.}, 129(6):1769--1773, 2001.

\bibitem{Gu15}
Purvi Gupta.
\newblock Lower-dimensional {F}efferman measures via the {B}ergman kernel.
\newblock {\em Contemp. Math. (Proceedings of the Conference on Analysis and
  Geometry in Several Complex Variables, Doha, Qatar, January 2015)}, to
  appear.

\bibitem{Ho65}
Lars H{\"o}rmander.
\newblock ${L}^2$ estimates and existence theorems for the operator.
\newblock {\em Acta Math.}, 113(1):89--152, 1965.

\bibitem{Jo2014}
Fritz John.
\newblock Extremum problems with inequalities as subsidiary conditions.
\newblock In {\em Traces and Emergence of Nonlinear Programming}, pages
  197--215. Springer, 2014.

\bibitem{Na12}
Fedor Nazarov.
\newblock The {H}{\"o}rmander proof of the {B}ourgain-{M}ilman theorem.
\newblock In {\em Geometric Aspects of Functional Analysis}, pages 335--343.
  Springer, 2012.

\bibitem{Oh84}
Takeo Ohsawa.
\newblock Boundary behavior of the {B}ergman kernel function on pseudoconvex
  domains.
\newblock {\em Publ. Res. Inst. Math. Sci.}, 20(5):897--902, 1984.

\bibitem{Sa88}
Saburou Saitoh.
\newblock {F}ourier-{L}aplace transforms and the {B}ergman spaces.
\newblock {\em Proc. Amer. Math. Soc.}, 102(4):985--992, 1988.

\bibitem{Sc91}
Carsten Sch{\"u}tt.
\newblock The convex floating body and polyhedral approximation.
\newblock {\em Israel J. Math.}, 73(1):65--77, 1991.

\bibitem{ScWe90}
Carsten Sch{\"u}tt and Elisabeth Werner.
\newblock The convex floating body.
\newblock {\em Math. Scand.}, 66:275--290, 1990.

\bibitem{ScWe92}
Carsten Sch{\"u}tt and Elisabeth Werner.
\newblock The convex floating body of almost polygonal bodies.
\newblock {\em Geom. Dedicata}, 44(2):169--188, 1992.

\bibitem{ScWe94}
Carsten Sch{\"u}tt and Elisabeth Werner.
\newblock Homothetic floating bodies.
\newblock {\em Geom. Dedicata}, 49(3):335--348, 1994.

\bibitem{St06}
Alina Stancu.
\newblock The floating body problem.
\newblock {\em Bull. Lond. Math. Soc.}, 38(5):839--846, 2006.

\bibitem{Zw}
W{\l}odzimierz Zwonek.
\newblock {\em Completeness, {R}einhardt domains and the method of complex
  geodesics in the theory of invariant functions}.
\newblock Polska Akademia Nauk, Instytut Matematyczny, 2000.

\end{thebibliography}
\bibliographystyle{plain}

\end{document}